\documentclass[10pt]{amsart}

\usepackage{amssymb}
\usepackage{enumerate}
\usepackage{subfigure}
\usepackage{setspace}
\usepackage{graphicx}

\numberwithin{equation}{section}

\newtheorem{thm}{Theorem}[section]

\newtheorem{lem}[thm]{Lemma}

\theoremstyle{definition}

\newcounter{countcases}

\newcommand{\PBS}[1]{\let\temp=\\#1\let\\=\temp}  

\numberwithin{figure}{section}

\begin{document}
\author{Brendon Rhoades}
\address{Brendon Rhoades, Department of Mathematics, Massachusetts Institute of Technology, Cambridge, MA, 02139}
\email{brhoades@math.mit.edu}
\title[Enumeration of Connected Catalan Objects by Type]
{Enumeration of Connected Catalan Objects by Type}

\bibliographystyle{../dart}

\date{\today}

\begin{abstract}
Noncrossing set partitions, nonnesting set partitions, Dyck paths,
and rooted plane trees are four classes of Catalan objects which carry
a notion of type.  There exists a product formula which enumerates these
objects according to type.  We define a notion of `connectivity' for 
these objects and prove an analogous product formula which counts
connected objects by type.  
Our proof of this product formula is combinatorial and bijective.
We extend this to a product formula which
counts objects with a fixed type and number of connected components.
We relate our product formulas to symmetric functions
arising from parking functions.  We close by presenting an alternative
proof of our product formulas communicated to us by Christian
Krattenthaler \cite{KrattComm} which uses generating functions and Lagrange
inversion.
\end{abstract}
\maketitle
\section{Introduction}

The \emph{Catalan numbers} $C_n = \frac{1}{n+1}{2n \choose n}$ are 
among the most important sequences of numbers in combinatorics.  To name
just a few examples (see \cite{StanEC2} for many more), the number
$C_n$ counts 123-avoiding 
permutations in $\mathfrak{S}_n$, 
Dyck paths of length $2n$, standard Young
tableaux of shape $2 \times n$, noncrossing or nonnesting
set partitions of $[n]$, and rooted plane trees with $n+1$ vertices.  

Certain families of Catalan objects
come equipped with a natural notion of \emph{type}.  For example, the 
type of a noncrossing set partition of $[n]$ is the sequence
$\mathbf{r} = (r_1, \dots, r_n)$, 
where $r_i$ is the number of blocks of size $i$.  
In the cases of noncrossing/nonnesting set partitions of $[n]$, Dyck paths
of length $2n$, and plane trees on $n+1$ vertices, there exists a nice 
product formula (Theorem 1.1) which counts Catalan objects with fixed type
$\mathbf{r}$.  These four classes of Catalan objects also carry a notion
of \emph{connectivity}.  In this paper we give a product formula which
counts these objects with a fixed type and a fixed number of connected
components.

\begin{figure}
\centering
\includegraphics[scale=.6]{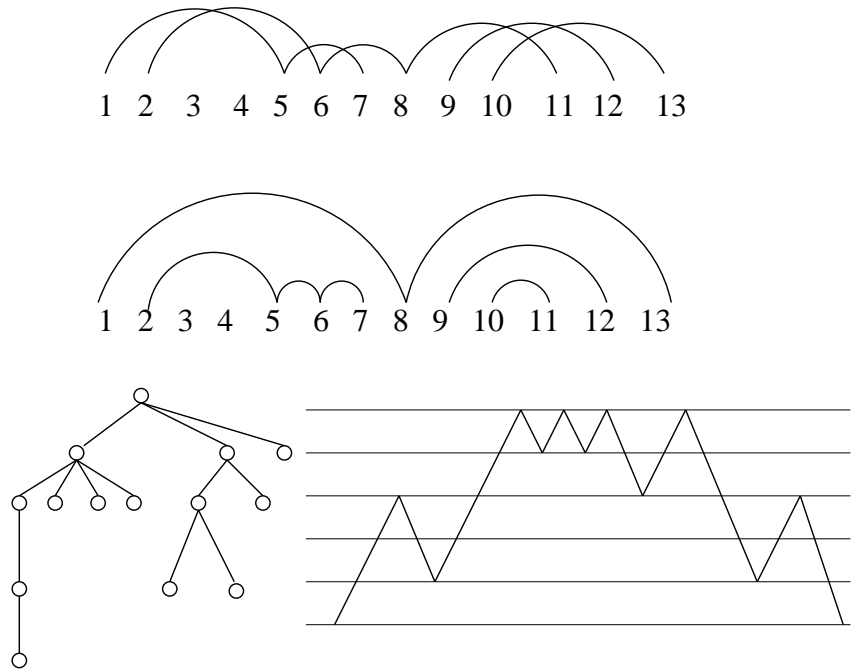}
\caption{A connected nonnesting partition of $[13]$,
a connected noncrossing partition of $[13]$, a plane tree with $14$ 
vertices with a terminal rooted twig, and a Dyck path of length $26$ with
no returns}
\end{figure}

The \emph{bump diagram} of a set partition $\pi$ of $[n]$ 
is obtained by drawing
the numbers $1$ through $n$ in a line and drawing an arc between $i$ and $j$
with $i < j$ if $i$ and $j$ are blockmates in $\pi$ and there does not
exist $k$ with $i < k < j$ such that $i,k,$ and $j$ are blockmates in $\pi$.
The set partition $\pi$ is \emph{noncrossing} if the bump diagram of $\pi$
has no crossing arcs or, equivalently, if there do not exist $a < b < c < d$
with $a,c$ in a block of $\pi$ and $b,d$ in a different block of $\pi$.
Similarly, the set partition $\pi$ is \emph{nonnesting} if the bump diagram
of $\pi$ contains no nesting arcs, that is, no pair of arcs of the form 
$ad$ and $bc$ with $a < b < c < d$.  As above, the \emph{type} of any 
set partition $\pi$ of $[n]$ is the sequence $(r_1, \dots, r_n)$, where
$r_i$ is the number of blocks in $\pi$ of size $i$.  
The set partition $\pi$ is called \emph{connected} if there does not 
exist an index $i$ with $1 \leq i \leq n-1$ such that there are no
arcs connecting the intervals $[1,i]$ and $[i+1,n]$ in the bump 
diagram of $\pi$.  The set partition $\pi$ is said to have 
\emph{$m$ connected components} if there exist numbers
$1 \leq i_1 < i_2 < \dots < i_{m-1} \leq n$ such that the restriction of 
the bump diagram of $\pi$ to each of the intervals 
$[1,i_1], [i_1+1,i_2], \dots, [i_{m-1}+1,n]$ is a connected set partition.

The bump diagram
of the noncrossing partition 
$\{1, 8, 13 / 2, 5, 6, 7 / 3 / 4 / 9, 12 / 10, 11 \}$ of $[13]$ with type
$(2,2,1,1,0,\dots,0)$ is shown in the middle of Figure 1.1.  The 
bump diagram of the nonnesting partition 
$\{1, 5, 7 / 2, 6, 8, 11 / 3 / 4 / 9,  12 / 10, 13 \}$ of $[13]$ with type
$(2,2,1,1,0,\dots,0)$ is shown in the top of Figure 1.1.  Both of 
these set partitions are connected.  The set partition
$\{ 1, 4 / 2, 3 / 5 / 6, 7, 8 \}$ is a noncrossing partition of $[8]$ with
$3$ connected components and type $(1,2,1,0,\dots,0)$.

A \emph{Dyck path} of length $2n$ is a lattice path in $\mathbb{Z}^2$
starting at $(0,0)$ and ending at $(2n, 0)$ which contains steps of 
the form $U = (1,1)$ and $D = (1,-1)$ and never goes below the 
$x$-axis.  An \emph{ascent} in a Dyck path is a maximal sequence
of $U$-steps.  The \emph{ascent type} of a Dyck path of length $2n$ is
the sequence $(r_1, \dots, r_n)$, where $r_i$ is the number of 
ascents of length $i$.  
A \emph{return} of a Dyck path of length $2n$ is a lattice point
$(m,0)$ with $0 < m < 2n$ which is contained in the Dyck path.

The ascent type of the Dyck path of length $26$
shown on the lower right of Figure 1.1 is $(2,2,1,1,0,\dots,0)$.  This
Dyck path has no returns.  The Dyck path
$UUDDUDUUDUDD$ has length $12$, ascent type $(2,2,0,\dots,0)$, and
$2$ returns.

A \emph{(rooted) plane tree} is a graph $T$ defined recursively as follows.
A distinguished vertex is called the 
\emph{root} of $T$ and the vertices of $T$
excluding the root are partitioned into an 
\emph{ordered} list of $k$ sets 
$T_1, \dots, T_k$, each of which is a plane tree.
Given a plane
tree $T$ on $n+1$ vertices, the \emph{downdegree sequence} of $T$ is the 
sequence $(r_0, r_1, \dots, r_n)$, where $r_i$ in the number of vertices
$v \in T$
with $i$ neighbors
further from the root than $v$.  
If $T$ is a plane tree with $n+1$ vertices, there
exists a labeling of the vertices of $T$ with $[n+1]$ called \emph{preorder}
(see \cite{StanEC1} for the precise definition).  The plane tree
$T$ with $n+1$ vertices is said to have a \emph{terminal rooted twig} if 
the vertex labeled $n+1$ is attached to the root.  A \emph{plane forest}
$F$ is an \emph{ordered} list of plane trees $F = (T_1, \dots, T_k)$.
The \emph{downdegree sequence} of a plane forest $F$ is the sum of
the downdegree sequences of its constituent trees.

The downdegree sequence of the plane tree with $14$ 
vertices shown on the lower left of Figure 1.1 is
$(8,2,2,1,1,0,\dots,0)$.  This plane tree has a terminal rooted 
twig. 

In order to avoid enforcing conventions such as $\frac{(-1)!}{(-1)!} = 1$ 
in the `degenerate' cases of our product formulas, we adopt the 
following notation of Zeng \cite{Zeng}.  Given any vectors 
$\mathbf{r} = (r_1, \dots, r_n), \mathbf{v} = (v_1, \dots, v_n) \in \mathbb{N}^n$, set
$| \mathbf{r} | := r_1 + \cdots + r_n$,
$\mathbf{r}! := r_1! r_2! \cdots r_n!$, and
$\mathbf{r} \cdot \mathbf{v} := r_1 v_1 + \cdots + r_n v_n$.
Let $x$ be a variable and for any 
vectors $\mathbf{r}, \mathbf{v} \in \mathbb{N}^n$ let 
$A_{\mathbf{r}}(x;\mathbf{v}) \in \mathbb{R}[x]$ be the polynomial
\begin{equation}
A_{\mathbf{r}}(x;\mathbf{v}) = 
\frac{x}{x + \mathbf{r} \cdot \mathbf{v}} 
\frac{(x + \mathbf{r} \cdot \mathbf{v})_{|\mathbf{r}|}}{\mathbf{r}!},
\end{equation}
where $(y)_k = y (y-1) \cdots (y-k+1)$ is a falling factorial.  Zeng
used the polynomials $A_{\mathbf{r}}(x;\mathbf{v})$ to prove various
convolution identities involving multinomial coefficients.

\begin{thm}
Let $n \geq 1$, let $\mathbf{v} = (1,2,\dots, n)$, and suppose 
$\mathbf{r} = (r_1, \dots, r_n) \in \mathbb{N}^n$ satisfies 
$\mathbf{r \cdot v} = n$.  

The polynomial evaluation
$A_{\mathbf{r}}(1; \mathbf{v}) = [A_{\mathbf{r}}(x; \mathbf{v})]_{x = 1}$
is equal to 
\footnote{This polynomial evaluation can also be expressed as
$\frac{n!}{(n-|\mathbf{r}|+1)!\mathbf{r}!}$.}
the cardinality of: 
\\
1.  the set of noncrossing partitions of $[n]$ of type
$\mathbf{r}$; \\
2.  the set of nonnesting partitions of $[n]$ of type
$\mathbf{r}$; \\
3.  the set of Dyck paths of length $2n$ with ascent
type $\mathbf{r}$; \\
4.  the set of plane trees with $n+1$
vertices and with downdegree sequence 
$(n-|\mathbf{r}|+1, r_1, \dots, r_n)$. 
\end{thm}

Part 1 of Theorem 1.1 is due to Kreweras \cite[Theorem 4]{Kreweras}.  
A type-preserving
bijection showing
the equivalence of Parts 1 and 2 was discovered by Athanasiadis 
\cite[Theorem 3.1]{AthNC}.  A similar bijection showing the equivalence
of Parts 1 and 3 was proven by Dershowitz and Zaks \cite{DZ}.  Armstrong
and Eu \cite[Lemma 3.2]{ArmEu} give an example of a bijection proving the 
equivalence of Parts 1 and 4.
  
The rest of this paper is organized as follows.
In Section 2 we
prove an analogous product formula (Theorem 2.2)
which counts connected objects according to type.  
The proof of Theorem 2.2 is bijective and relies on certain properties
of words in monoids.
We extend this
result to another product formula (Theorem 2.3) which counts objects
which have a fixed number of connected components according to type.  These
product formulas have found a geometric application in \cite{ArmRho} where
they are used to count regions of hyperplane arrangements related to 
the Shi arrangement according to `ideal dimension' in the sense
of Zaslavsky \cite{ZIdeal}.  We then apply our product formulas
to the theory of symmetric functions, refining a formula of
Stanley
\cite{StanPark}.  In Section
3 we present an alternative proof of Theorem 2.3
communicated to us by Christian Krattenthaler \cite{KrattComm}
which uses generating functions and Lagrange inversion.

\section{Main Results}

The proofs of our product formulas will rest on a lemma about
words in monoids which can be viewed as a `connected analog'
of the `cycle lemma' due to Dvoretzky and Motzkin \cite{DM}
(see also \cite{DZCycle}).  
For a more leisurely introduction
to this material, see \cite{StanEC2}.

Let $\mathcal{A}$ 
denote the infinite alphabet $\{x_0, x_1, x_2, \dots \}$ and let 
$\mathcal{A}^*$ 
denote the free (noncommutative) monoid generated by 
$\mathcal{A}$.  Denote
the empty word by $e \in \mathcal{A}^*$.  
The \emph{weight function} is the monoid
homomorphism $\omega: \mathcal{A}^* \rightarrow (\mathbb{Z}, +)$ induced by
$\omega(x_i) = i-1$ for all $i$.  We define a subset 
$\mathcal{B} \subset \mathcal{A}^*$ by
\begin{equation*}
\mathcal{B} = \{ w = w_1 \dots w_n \in \mathcal{A}^* \,|\, \omega(w) = 1, 
\text{$\omega(w_1 w_2 \dots w_j) > 0$ for $1 \leq j \leq n$} \}.
\end{equation*}
That is, a word $w \in \mathcal{A}^*$ is contained in 
$\mathcal{B}$ if and only if
it has weight $1$ and all of its nonempty prefixes have positive weight.
In particular, we have that $e \notin \mathcal{B}$.  

Given any word $w = w_1 \dots w_n \in \mathcal{A}^*$, 
a \emph{conjugate} of $w$
is an element of $\mathcal{A}^*$ of the form 
$w_i w_{i+1} \dots w_n w_1 w_2 \dots w_{i-1}$ for some $1 \leq i \leq n$
(this is the monoid-theoretic analog of conjugation in groups).  We have 
the following result concerning conjugacy classes of elements of 
$\mathcal{B}$.  It
is our `connected analog of' \cite[Lemma 5.3.7]{StanEC2} and is an analog of
the `cycle lemma' in tree enumeration.

\begin{lem}
A word $w \in \mathcal{A}^*$ is conjugate to an element of 
$\mathcal{B}$ if and only if
$\omega(w) = 1$, in which case $w$ is conjugate to a unique element of
$\mathcal{B}$ and the 
conjugacy class of $w$ has size equal to the length of $w$.
\end{lem}

\begin{proof}
Let $w \in \mathcal{B}$ have length $n$ and suppose the conjugacy
class of $w$ has size $k | n$.  Then we can write
$w = v^{n/k}$ for some $v \in \mathcal{A}^*$.  Since $v$ is a nonempty
prefix of $w$, we have $\omega(v) > 0$ and the fact that
$1 = \omega(w) = \frac{n}{k} \omega(v)$ forces $k = n$.

Since conjugation does not affect weight,
every element $w$ of the conjugacy class of an element of 
$\mathcal{B}$
satisfies
$\omega(w) = 1$.  

Suppose that $w \in \mathcal{A}^*$ satisfies $\omega(w) = 1$. 
We show that $w$ is conjugate to an element of $\mathcal{B}$.  
The proof of this fact
breaks up into three cases depending on the letters
which occur in $w$.

\noindent
\emph{Case 1:  $w$ contains no occurrences of $x_0$.}
Since $\omega(w) = 1$, in this case $w$ must be of the form 
$x_1 \dots x_1 x_2 x_1 \dots x_1$ and $w$ is conjugate to
$x_2 x_1 \dots x_1 \in \mathcal{B}$.  

\noindent
\emph{Case 2:  $w$ contains at most one occurrence of a letter 
other than $x_0$.}
In this case, the
condition $\omega(w) = 1$ forces a conjugate of 
$w$ to be of the form $x_s x_0^{s-2} \in \mathcal{B}$ for
some $s > 1$.  

\noindent
\emph{Case 3: $w$ at least one occurrence of $x_0$ and at least two
occurrences of letters other than $x_0$.}
We claim that
there exists a conjugate $w'$ of $w$ of the form 
$w' = x_{s+1} x_0^s v$ for some $s \geq 0$.  If this were not the case, 
consider the word $w$ written around a circle.  Every maximal contiguous
string of $x_0$'s in $w$ of length $\ell$ must be preceded by a letter
of the form $x_s$ for some $s > \ell + 1$.  The weight
of any such contiguous string taken together with its 
preceding letter is
$\omega(x_s x_0^{\ell}) = s - 1 - \ell > 0$.  Since $\omega(w) = 1$, it
follows that $w$ has a conjugate of the form $x_s x_0^{s-2}$, which
contradicts our assumption that $w$ has at least two occurrences of
a letter other
than $x_0$.  
Let $w'$ be a conjugate of $w$ of the form $w' = x_{s+1} x_0^s v$ for
some $v \in \mathcal{A}^*$.  
Since $1 = \omega(w) = \omega(w') = s -s + \omega(v)$,
by induction 
on length
we can assume that a conjugate of $v$ is contained in
$\mathcal{B}$.  Say that $v = yz$ such that $zy \in \mathcal{B}$ with 
$z \neq e$.  Then $z x_{s+1} x_0^s y$ is a conjugate of
$w' = x_{s+1} x_0^s yz$ satisfying $z x_{s+1} x_0^s y \in \mathcal{B}$.
\end{proof}

Let $\mathcal{B}^*$ denote the submonoid of $\mathcal{A}^*$
generated by $\mathcal{B}$.
In view of \cite[Lemma 5.3.7]{StanEC2}, it is tempting to guess that any 
element $w \in \mathcal{A}^*$ obtained by permuting the letters of 
an element of $\mathcal{B}^*$ is itself conjugate to an element of 
$\mathcal{B}^*$.  
However, this is false.  For example, the element
$x_3 x_0 x_2 = (x_3 x_0) (x_2) \in \mathcal{A}^*$ is contained
in $\mathcal{B}^*$
but $x_3 x_2 x_0$ has no conjugate
in $\mathcal{B}^*$.
(However, the analog of \cite[Lemma 5.3.6]{StanEC2} does hold in this
context - the monoid $\mathcal{B}^*$ is very pure.)
Lemma 2.1 is the key tool we will use in proving our connected analog
of Theorem 1.1.

\begin{thm}
Let $n \geq 1$, let $\mathbf{v} = (1,2,\dots,n)$, and suppose
$\mathbf{r} = (r_1, \dots, r_n) \in \mathbb{N}^n$ satisfies
$\mathbf{r \cdot v} = n$.  

The polynomial evaluation
$-A_{\mathbf{r}}(-1;\mathbf{v}) = [-A_{\mathbf{r}}(x;\mathbf{v})]_{x = -1}$
is equal to
\footnote{In the case where $n > 1$ and $\mathbf{r} \neq (n,0,\dots,0)$, this
can also be expressed as $\frac{(n-2)!}{(n-|\mathbf{r}|-1)! \mathbf{r}!}$.}
the cardinality of:\\
1.  the set of connected noncrossing partitions of $[n]$ of type
$\mathbf{r}$; \\
2.  the set of connected nonnesting partitions of $[n]$ of type
$\mathbf{r}$; \\
3.  the set of Dyck paths of length $2n$ with no returns and ascent
type $\mathbf{r}$; \\
4.  the set of plane trees with a terminal rooted twig and $n+1$
vertices with downdegree sequence $(n-|\mathbf{r}|+1, r_1, \dots, r_n)$. \\
\end{thm}
\begin{proof}
The line of reasoning which we follow here should be compared to that in
\cite[Chapter 5]{StanEC2}.  

Observe first that when $\mathbf{r} = (n,0,\dots,0)$ we have that 
\begin{equation}
-A_{\mathbf{r}}(-1;\mathbf{v}) = \begin{cases}
1 & \text{if $n = 1$,} \\
0 & \text{if $n > 1$.}
\end{cases}
\end{equation}
This is in agreement with the relevant set cardinalities, so from now on
we assume that $n > 1$ and $\mathbf{r} \neq (n,0,\dots,0)$.
Let $\mathcal{B}(\mathbf{r})$ denote the set of 
length $n-1$
words $w \in \mathcal{B}$
with $n-|\mathbf{r}|-1$ $x_0$'s, $r_1$ $x_1$'s, $\dots$, and $r_n x_n$'s.  
By Lemma 2.1, we have that
\begin{equation}
|\mathcal{B}(\mathbf{r})| = \frac{1}{n-1}{n-1 \choose 
n - |\mathbf{r}| - 1, r_1, \dots, r_n} = -A_{\mathbf{r}}(-1;\mathbf{v}),
\end{equation}
where the second equality follows from the definition of 
$A_{\mathbf{r}}(x;\mathbf{v})$.  Therefore, it suffices to biject
each of the sets in Parts 1-4 with the set 
$\mathcal{B}(\mathbf{r})$.  We present a bijection in each case.

1.  Let $NC(\mathbf{r})$ be the set of noncrossing partitions we wish
to enumerate.  Given any partition $\pi$ of $[n]$, define a word
$\psi(\pi) = w_1 w_2 \dots w_{n-1} \in \mathcal{A}^*$ as follows.  For
$1 \leq i \leq n-1$, if $i$ is the minimal element of a block of $\pi$, 
let $w_i = x_j$ where $j$ is the size of the block containing $i$.  
Otherwise, let $w_i = x_0$.  For example, if $\pi$ is the connected
nonnesting partition of $[13]$ shown on the top of Figure 1.1,
we have that 
$\psi(\pi) = x_3 x_4 x_1 x_1 x_0 x_0 x_0 x_0 x_2 x_2 x_0 x_0$.
It is easy to see that the mapping $\pi \mapsto \psi(\pi)$ sets up a
bijection between set partitions in $NC(\mathbf{r})$ and words in
$\mathcal{B}(\mathbf{r})$.

2.  Let $NN(\mathbf{r})$ be the set of nonnesting partitions we wish
to enumerate.  
It is easy to verify (see 
\cite[Solution to Exercise 5.44]{StanEC2}) that the map
$\psi$ from the proof of Part 1 restricts to a bijection between
$NN(\mathbf{r})$ and $\mathcal{B}(\mathbf{r})$.

3.  Let $\mathbb{D}$ be a Dyck path with no returns of length $2n$ 
and ascent type $\mathbf{r}$.  Define a length $n-1$ word
$\delta(\mathbb{D}) \in \mathcal{A}^*$ as follows.  Let 
$w_1 w_2 \dots w_n \in \mathcal{A}^*$ be the word obtained by reading
$\mathbb{D}$ from left to right, replacing every ascent of length
$i$ with $x_i$ and replacing every maximal contiguous sequence of downsteps
of length $\ell$ with $x_0^{\ell-1}$.  Set 
$\delta(\mathbb{D}) := w_1 w_2 \dots w_{n-1}$.  For example, if $\mathbb{D}$
is the Dyck path shown in Figure 1.1 we have that
$\delta(\mathbb{D}) = x_3 x_0 x_4 x_1 x_1 x_0 x_2  x_0^3 x_2 x_0$.  It is
easy to verify that 
$\delta(\mathbb{D}) \in \mathcal{B}(\mathbf{r})$ 
and that the map
$\mathbb{D} \mapsto \delta(\mathbb{D})$ sets up a bijection between
Dyck paths with no returns of length $2n$ and ascent type 
$\mathbf{r}$ to $\mathcal{B}(\mathbf{r})$. 

4.  For $T$ be a plane tree on $n+1$ vertices with a terminal rooted twig,
let $w_1 w_2 \dots w_{n+1} \in \mathcal{A}^*$ be the word obtained by
setting $w_i = x_j$, where $j$ is the downdegree of the $i^{th}$ vertex
of $T$ in preorder.  Since $T$ has a terminal rooted twig, we have
$w_n = w_{n+1} = x_0$.  Set 
$\chi(T) := w_1 w_2 \dots w_{n-1} \in \mathcal{A}^*$.
For example, if $T$ is the tree shown in Figure 1.1, we have
that $\chi(T) = x_3 x_4 x_1 x_1 x_0 x_0 x_0 x_0 x_2 x_2 x_0 x_0$.
The mapping $T \mapsto \chi(T)$ sets up a bijection between the
set of trees of interest and $\mathcal{B}(\mathbf{r})$.
\end{proof}

An alternative proof of Parts 1 and 2 of Theorem 2.2 which relies
on a product formula enumerating noncrossing partitions by `reduced 
type' due to Armstrong and Eu \cite{ArmEu} (which in turn relies on
the original enumeration of noncrossing partitions by type due to
Kreweras) can be found in \cite{ArmRho}.

It is natural to ask if the formula in Theorem 2.2 can be generalized to
the case of multiple connected components.  The answer is `yes', and 
to avoid enforcing nonstandard conventions in degenerate cases
we will again
state the relevant product formula in terms of a polynomial specialization.
Suppose that $\mathbf{r}, \mathbf{v} \in \mathbb{N}^n$ and $1 \leq m \leq 
|\mathbf{r}|$.  We define the polynomial $A^{(m)}_{\mathbf{r}}(x;\mathbf{v})
\in \mathbb{R}[x]$ by
\begin{equation}
A^{(m)}_{\mathbf{r}}(x;\mathbf{v}) = \frac{(|\mathbf{r}|-1)!}
{(|\mathbf{r}|-m)!}
\frac{x}{x + \mathbf{r \cdot v}}
\frac{(x + \mathbf{r \cdot v})_{|\mathbf{r}|-m+1}}{\mathbf{r!}}.
\end{equation}
Observe that in the case $m = 1$ we have $A^{(1)}_{\mathbf{r}}(x;\mathbf{v}) 
= A_{\mathbf{r}}(x;\mathbf{v})$.

\begin{thm}
Let $n \geq m \geq 1$ and let $\mathbf{v} = (1,2,\dots,n) \in \mathbb{N}^n$.
Suppose that $\mathbf{r} = (r_1, \dots, r_n) \in \mathbb{N}^n$ satisfies 
$\mathbf{r \cdot v} = n$ and $m \leq |\mathbf{r}|$.  

The polynomial 
evaluation $-A^{(m)}_{\mathbf{r}}(-m;\mathbf{v}) 
= [-A^{(m)}_{\mathbf{r}}(x;\mathbf{v})]_{x = -m}$ 
is equal to
\footnote{In the case where $n > m$ and $\mathbf{r} \neq (n,0,\dots,0)$, this
can also be expressed as 
$\frac{m(n-m-1)!(|\mathbf{r}|-1)!}{(n-|\mathbf{r}|-1)!(|\mathbf{r}|-m)! \mathbf{r}!}$.}
the cardinality of:\\
1. the set of noncrossing partitions of $[n]$ with exactly $m$ 
connected components of type $\mathbf{r}$;\\
2. the set of nonnesting partitions of $[n]$ with exactly $m$
connected components of type $\mathbf{r}$;\\
3. the set of Dyck paths of length $2n$ with exactly $m-1$ returns of 
ascent type $\mathbf{r}$;\\
4. the set of plane forests with $n + m$ vertices and exactly $m$ trees with
downdegree sequence $(n-|\mathbf{r}|+m,r_1,\dots,r_n)$ such that
every tree has a terminal rooted twig. 
\end{thm}
\begin{proof}
In light of Theorem 2.2, it suffices to prove Part 1.
The polynomial $A_{\mathbf{r}}(mx;\mathbf{v})$ can be obtained via the
following convolution-type identity which follows from 
a result of Raney
\cite[Theorems 2.2,
2.3]{Raney}
and induction:
\begin{equation}
\sum_{\mathbf{r^{(1)}} + \cdots + \mathbf{r^{(m)}} = \mathbf{r}}
A_{\mathbf{r^{(1)}}}(x;\mathbf{v}) \cdots A_{\mathbf{r^{(m)}}}(x;\mathbf{v}) =
A_{\mathbf{r}}(mx;\mathbf{v}),
\end{equation}
where $\mathbf{r^{(i)}} \in \mathbb{N}^n$ for all $i$.
Let $\mathbf{0} \in \mathbb{N}^n$ be the zero vector.
By Theorem 2.2 and the fact that $A_{\mathbf{0}}(x;\mathbf{v}) = 1$, 
we can set $x = -1$ to obtain
\begin{equation}
\sum_{k=1}^m (-1)^k {m \choose k} C(n,k,\mathbf{r}) = 
A_{\mathbf{r}}(-m;\mathbf{v}),
\end{equation}
where $C(n,k,\mathbf{r})$ denotes the number of noncrossing partitions of
$[n]$ with exactly $k$ connected components and type $\mathbf{r}$.  By the
Principle of Inclusion-Exclusion (see \cite{StanEC1}), it follows that 
\begin{equation}
C(n,m,\mathbf{r}) = \sum_{k=1}^m (-1)^k {m \choose k} 
A_{\mathbf{r}}(-k;\mathbf{v}).
\end{equation}
Therefore, it suffices to show that the right hand side of  
Equation 2.6
is equal to $-A^{(m)}_{\mathbf{r}}(-m;\mathbf{v})$.  We sketch 
this verification here for the case $m, |\mathbf{r}| < n$;  the other
degenerate cases are left to the reader.  

We start with the following binomial coefficient identity:
\begin{equation}
\sum_{k=1}^m (-1)^{k+1} k {m \choose k} {n - k - 1 \choose |\mathbf{r}| - 1}
= m {n - m - 1 \choose n - |\mathbf{r}| - 1}.
\end{equation} 
This 
identity can be
obtained by comparing like powers of $x$ on both sides
of the equation 
$r(1+x)^{r+s-1} = (1+x)^s \frac{d}{dx}(1+x)^r = (1+x)^s 
({r \choose 1} + 2 {r \choose 2} x + 3 {r \choose 3} x^2 + \cdots)$.  
Multiplying both sides 
of Equation 2.7
by $\frac{(|\mathbf{r}|-1)!}{\mathbf{r}!}$ and
using the definition of $A_{\mathbf{r}}(x;\mathbf{v})$ we obtain
\begin{equation}
\sum_{k=1}^m (-1)^{k} {m \choose k} A_{\mathbf{r}}(-k;\mathbf{v}) = 
\frac{m (n-m-1)! (|\mathbf{r}|-1)!}
{(n-|\mathbf{r}|-1)!(|\mathbf{r}|-m)!\mathbf{r}!}.
\end{equation}
The right hand side of Equation 2.8 is equal to 
$-A_{\mathbf{r}}^{(m)}(-m;\mathbf{v})$.
\end{proof}

We close this section 
by relating the product formulas in this paper to
Frobenius characters arising from the theory of parking functions.
For $n \geq 0$ a \emph{parking function of length $n$} is a sequence
$(a_1, \dots, a_n)$ of positive integers whose nondecreasing rearrangement
$(b_1, \dots, b_n)$ satisfies $b_i \leq i$ for all $i$.  A nondecreasing
parking function is called \emph{primitive} and primitive parking
functions of length $n$ are in an obvious bijective correspondence
(see \cite{ArmEu}) with Dyck paths of length $2n$.  The \emph{type} of
a parking function is the ascent type of its nondecreasing rearrangement.
A parking function $(a_1, \dots, a_n)$ will be said to have \emph{$m$
returns} if its nondecreasing rearrangement has $m$ returns when viewed
as a Dyck path.

The symmetric group $\mathfrak{S}_n$ acts on the set of parking
functions of length $n$.  Stanley \cite{StanPark} computed the Frobenius
character of this action with respect to the standard bases
(monomial, homogeneous, elementary, power sum, and Schur) of the 
ring of symmetric functions.  To compute this character in the basis
$\{ h_{\lambda} \}$ of complete homogeneous symmetric functions, he 
observed that every orbit $\mathcal{O}$ of this action 
contains a unique primitive parking
function $(b_1, \dots, b_n)$
and that the Frobenius 
character of 
the action of $\mathfrak{S}_n$ on
$\mathcal{O}$ is $h_{\lambda}$, where
$\lambda = (1^{r_1} 2^{r_2} \dots n^{r_n})$ and
$(r_1, \dots, r_n)$ is the type of $(b_1, \dots, b_n)$.  By applying
the formula in Theorem 1.1, one immediately gets the expansion
\begin{equation}
\mathrm{Frob}(P_n) = \sum_{\lambda \vdash n} \frac{n!}
{(n-|\mathbf{r}(\lambda)|+1)!
\mathbf{r}(\lambda)!} h_{\lambda},
\end{equation}
where $\mathrm{Frob}$ is the Frobenius characteristic map, 
$P_n$ is the permutation module for the action of $\mathfrak{S}_n$ on
parking functions of length $n$,
$r_i(\lambda)$ is
the multiplicity of $i$ in $\lambda$ for $1 \leq i \leq n$, and
$\mathbf{r}(\lambda) = (r_1(\lambda), \dots, r_n(\lambda))$.  (See \cite{ArmEu}
for a nonhomogeneous generalization of this formula.)  

Using the same reasoning as in \cite{StanPark} we can compute the 
Frobenius characters of other modules related to parking functions.
In particular, for $m > 0$, the symmetric group
$\mathfrak{S}_n$ acts on the set of parking functions of length $n$
with $m-1$ returns.  Let $P^{(m)}_n$ be the permutation module
corresponding to this action, so that $P_n \cong \bigoplus_{m=0}^{n-1}
P_n^{(m)}$.  Applying Theorem 2.3, we have that the Frobenius
character of this module is
\begin{equation}
\mathrm{Frob}(P^{(m)}_n) = \sum_{\lambda \vdash n} 
-A^{(m)}_{\mathbf{r}(\lambda)}(-m;\mathbf{v}) h_{\lambda},
\end{equation}
where $\mathbf{v} = (1,2,\dots,n) \in \mathbb{N}^n$.

\section{Proof of Theorem 2.3 using Lagrange Inversion}
In this section we outline an alternative proof of Theorem 2.3 using 
generating functions and Lagrange inversion which was pointed out to the 
author by Christian Krattenthaler \cite{KrattComm}.  This method has
the advantage of immediately proving Theorem 2.3 without first proving
the single connected component case of Theorem 2.2.  We only handle
the case of noncrossing partitions.

Let $y = \{ y_1, y_2 , \dots, \}$ 
and $z$ be
commuting variables.  If $\pi$ is a noncrossing partition of $[n]$ for
$n \geq 0$, the \emph{weight} of $\pi$ is the monomial 
\begin{equation}
\mathrm{wt}(\pi) = z^n \prod_{i \geq 1} y_i^{r_i(\pi)},
\end{equation}
where $r_i(\pi)$ is the number of blocks in $\pi$ of size $i$.  (The 
unique partition of $[0]$ has weight $1$.)  We define
$P(z) \in \mathbb{R}(y_1,y_2,\dots)[[z]]$ by grouping these monomials
together in a generating function.  That is,
\begin{equation}
P(z) = \sum_{\pi} \mathrm{wt}(\pi),
\end{equation}
where the sum is over all noncrossing partitions $\pi$.

Given any noncrossing partition $\pi$ of $[n]$ with $n > 1$, if the block
of $\pi$ containing $1$ has size $k$, drawing $\pi$ on a circle one obtains
$k$ (possibly empty) noncrossing partitions `between' each successive pair
of elements in this $k$ element block.  This combinatorial observation
yields the following formula:
\begin{equation}
P(z) = 1 + \sum_{k = 1}^{\infty} y_k z^k P(z)^k.
\end{equation}
Rearranging this expression, we get that
\begin{equation}
\frac{zP(z)}{1 + \sum_{k = 1}^{\infty} y_k z^k P(z)^k} = z,
\end{equation}
and therefore $zP(z)$ is the compositional inverse of 
$\frac{z}{X(z)}$, where $X(z) = 1 + \sum_{k = 1}^{\infty} y_k z^k$.  This
implies that
\begin{equation}
P\left(\frac{z}{X(z)}\right) = X(z).
\end{equation}

In order to prove Theorem 2.3, we need to keep track of the number of 
connected components of a noncrossing partition.  To do this,
let $C(z) \in \mathbb{R}(y_1,y_2,\dots)[[z]]$ the generating function
\begin{equation}
C(z) = \sum_{\pi} \mathrm{wt}(\pi),
\end{equation}
where the sum ranges over all \emph{connected}  
noncrossing
partitions of $[n]$
where $n \geq 1$.  It is immediate that the generating functions
$P(z)$ and $C(z)$ are related by
\begin{equation}
P(z) = \frac{1}{1 - C(z)}
\end{equation}
or equivalently,
\begin{equation}
C(z) = \frac{P(z)  - 1}{P(z)}.
\end{equation}
As in the first proof of Theorem 2.3, let $C(n,m,\mathbf{r})$ denote
the number of noncrossing partitions of $[n]$ with 
exactly $m$ connected components and type $\mathbf{r}$.  It is evident
that
\begin{equation}
C(z)^m = \left( \frac{P(z) - 1}{P(z)} \right)^m = \sum_{n \geq 0}
\sum_{\mathbf{r} \geq \mathbf{0}} C(n,m,\mathbf{r})y^{\mathbf{r}}z^n,
\end{equation}
where the inequality in the inner summation is componentwise and
$y^{\mathbf{r}} = y_1^{r_1} y_2^{r_2} \cdots$ if $\mathbf{r} 
= (r_1, r_2, \dots)$.

To find an expression for $C(n,m,\mathbf{r})$ it is enough to extract the
coefficient of $z^n y^{\mathbf{r}}$ from the generating 
function in Equation 3.9.  We use Lagrange inversion to do this.  
Set $F(z) := \frac{z}{X(z)}$, so that the compositional inverse of $F(z)$
is $F^{\langle -1 \rangle}(z) = z P(z)$.  Also set
$H(z) := \left( \frac{X(z) - 1}{X(z)} \right)^m$.  In light of Equation
3.5 we have the identity 
$\left( \frac{P(z) - 1}{P(z)} \right)^m = H (F^{\langle -1 \rangle}(z))$.
Let $\langle - \rangle$ denote taking a coefficient in a Laurent series.
Applying Lagrange inversion as in \cite[Corollary 5.4.3]{StanEC2} 
we get that 
\begin{align*}
C(n,m,\mathbf{r}) &= \langle z^n y^{\mathbf{r}} \rangle
H (F^{\langle -1 \rangle}(z)) \\
&=  \frac{1}{n} \langle z^{n-1} y^{\mathbf{r}} \rangle 
H'(z) \left( \frac{z}{F(z)} \right)^n  \\
&= \frac{1}{n} \langle z^{n-1} y^{\mathbf{r}} \rangle
m X(z)^{n-m-1} (X(z) - 1)^{m-1} X'(z) \\
&= \frac{m}{n} \langle z^{n-1} y^{\mathbf{r}} \rangle
(X(z) - 1)^{m-1} \sum_{\ell \geq 0} {n - m - 1 \choose \ell} 
(X(z) - 1)^{\ell} X'(z) \\
&= \frac{m}{n} \langle z^{n-1} y^{\mathbf{r}}  \rangle
(X(z) - 1)^{m-1} \sum_{\ell \geq 0} \frac{1}{m + \ell} {n - m - 1 \choose
\ell} \left( (X(z)-1)^{m + \ell} \right)', 
\end{align*} 
where all derivatives are partial derivatives with respect to $z$.
Suppose $\mathbf{r} = (r_1, r_2, \dots)$.  Taking the coefficient in the
bottom line yields the equality
\begin{equation}
C(n,m,\mathbf{r}) = \frac{m}{|\mathbf{r}|} {n-m-1 \choose |\mathbf{r}| - m}
{|\mathbf{r}| \choose r_1, r_2, \dots},
\end{equation}
which is equivalent to Part 1 of Theorem 2.3.

\section{Acknowledgments}
The author is grateful to Drew Armstrong, Christos Athanasiadis, 
Christian Krattenthaler,
Vic Reiner, and Richard Stanley for many helpful conversations.

\bibliography{../bib/my}

\begin{thebibliography}{10}

\bibitem{ArmEu}
{\sc D.~Armstrong and S.-P. Eu}.
\newblock Nonhomogeneous parking functions and noncrossing partitions.
\newblock {\em Electron. J. Combin.\/}, {\bf 15}, 1 (2008).

\bibitem{ArmRho}
{\sc D.~Armstrong and B.~Rhoades}.
\newblock The {S}hi arrangement and the {I}sh arrangement.
\newblock In progress, 2010.

\bibitem{AthNC}
{\sc C.~A. Athanasiadis}.
\newblock On noncrossing and nonnesting partitions for classical reflection
  groups.
\newblock {\em Elecron. J. Combin.\/}, {\bf 5} (1998).
\newblock R42.

\bibitem{DZ}
{\sc N.~Dershowitz and S.~Zaks}.
\newblock Ordered trees and {non-crossing} partitions.
\newblock {\em Discrete Math.\/}, {\bf 62} (1986).

\bibitem{DZCycle}
{\sc N.~Dershowitz and S.~Zaks}.
\newblock The cycle lemma and some applications.
\newblock {\em Europ. J. Combin.\/}, {\bf 11} (1990) pp. 35--40.

\bibitem{DM}
{\sc A.~Dvoretzky and T.~Motzkin}.
\newblock A problem of arrangements.
\newblock {\em Duke Math. J.\/}, {\bf 14} (1947) pp. 305--313.

\bibitem{KrattComm}
{\sc C.~Krattenthaler} (2010).
\newblock Personal communication.

\bibitem{Kreweras}
{\sc G.~Kreweras}.
\newblock Sur les partitions non {crois\'ees} d'un cycle.
\newblock {\em Discrete Math.\/}, {\bf 1}, 4 (1972) pp. 333--350.

\bibitem{Raney}
{\sc G.~Raney}.
\newblock Functional composition patterns and power series inversion.
\newblock {\em Trans. Amer. Math. Soc.\/}, {\bf 94} (1960) pp. 441--451.

\bibitem{StanEC1}
{\sc R.~Stanley}.
\newblock {\em {Enumerative Combinatorics}\/}, vol.~1.
\newblock Cambridge University Press, Cambridge (1997).

\bibitem{StanPark}
{\sc R.~Stanley}.
\newblock Parking functions and noncrossing partitions.
\newblock {\em Electron. J. Combin.\/}, {\bf 4}, 2 (1997).
\newblock R23.

\bibitem{StanEC2}
{\sc R.~Stanley}.
\newblock {\em {Enumerative Combinatorics}\/}, vol.~2.
\newblock Cambridge University Press, Cambridge (1999).

\bibitem{ZIdeal}
{\sc T.~Zaslavsky}.
\newblock Faces of a hyperplane arrangement enumerated by ideal dimension, with
  applications to plane, plaids and {S}hi.
\newblock {\em Geom. Decicata\/}, {\bf 98} (2003) pp. 63--80.

\bibitem{Zeng}
{\sc J.~Zeng}.
\newblock Multinomial convolution polynomials.
\newblock {\em Discrete Math.\/}, {\bf 160} (1996) pp. 219--228.

\end{thebibliography}
\end{document}